\documentclass[12pt,reqno]{article}

\usepackage[usenames]{color}
\usepackage{amssymb}
\usepackage{graphicx}
\usepackage{amscd}

\usepackage{amsthm}
\newtheorem{theorem}{Theorem}
\newtheorem{lemma}[theorem]{Lemma}
\newtheorem{proposition}[theorem]{Proposition}
\newtheorem{corollary}[theorem]{Corollary}

\theoremstyle{definition}
\newtheorem{definition}[theorem]{Definition}
\newtheorem{example}[theorem]{Example}

\newcommand{\sech}{\operatorname{sech}}
\newcommand{\polylog}{\operatorname{polylog}}

\usepackage[colorlinks=true,
linkcolor=webgreen, filecolor=webbrown,
citecolor=webgreen]{hyperref}

\definecolor{webgreen}{rgb}{0,.5,0}
\definecolor{webbrown}{rgb}{.6,0,0}

\usepackage{color}

\usepackage{float}

\usepackage{graphics,amsmath,amssymb}
\usepackage{amsfonts}
\usepackage{latexsym}
\usepackage{epsf}

\newcommand{\seqnum}[1]{\href{http://www.research.att.com/cgi-bin/access.cgi/as/~njas/sequences/eisA.cgi?Anum=#1}{\underline{#1}}}

\begin{document}

\begin{center}
\vskip 1cm{\LARGE\bf Eulerian-Dowling Polynomials as Moments, Using Riordan Arrays} \vskip 1cm \large
Paul Barry\\
School of Science\\
Waterford Institute of Technology\\
Ireland\\
\href{mailto:pbarry@wit.ie}{\tt pbarry@wit.ie}
\end{center}
\vskip .2 in

\begin{abstract} Using the theory of exponential Riordan arrays, we show that the Eulerian-Dowling polynomials are moments for a paramaterized family of orthogonal polynomials. In addition, we show that the related Dowling and the Tanny-Dowling polynomials are also moments for appropriate families of orthogonal polynomials. We provide continued fraction generating functions and Hankel transforms for these polynomials. \end{abstract}

\section{Introduction}

The authors Benoumhani \cite{Ben1, Ben2} and Rahmani \cite{Rahmani} have studied families of polynomials associated with the class of geometric lattices introduced by Dowling \cite{Dowling}. These are the Dowling, the Tanny-Dowling, and the Eulerian-Dowling polynomials.

In this note, we show that these polynomials can be studied within the context of exponential Riordan arrays. They then present themselves as moments of families of orthogonal polynomials \cite{Gautschi, Chihara, Szego}. We describe the coefficient arrays of the associated orthogonal polynomials in terms of exponential Riordan arrays, and exploiting the link between the tri-diagonal production matrices \cite{ProdMat1, ProdMat2, Peart} of the moment matrices and continued fractions \cite{Wall}, we determine the Hankel transforms \cite{Kratt, Layman} of these polynomials. The articles \cite{Euler1, Euler2} use similar techniques to describe the Eulerian polynomials and a special class of generalized Eulerian polynomials as moment sequences.

\section{Essentials of exponential Riordan arrays}
We briefly summarize the elements of the theory of exponential Riordan arrays \cite{Book, SGWW, Survey} that we will require.
An exponential Riordan array is defined by two power series
$$g(x)=1+g_1 \frac{x}{1!} + g_2 \frac{x^n}{n!}+\ldots = \sum_{n=0}^{\infty} g_n \frac{x^n}{n!},$$
and
$$f(x)=x+f_2 \frac{x^2}{2!} + f_3 \frac{x^3}{3!} + \ldots = \sum_{n=0}^{\infty} f_n \frac{x^n}{n!},$$
where $g_0=1$, $f_0=0$ and $f_1=1$. (It is possible to relax the conditions $f_1=1$ and $g_0=1$ to $f_1 \ne 0$ and $g_0 \ne 0$, but we do not do this here for simplicity).
The matrix $M$ with $(n,k)$-th element
$$\frac{n!}{k!} [x^n] g(x)f(x)^k$$ is then regarded as a concrete realization of the exponential Riordan array defined by the pair $(g(x), f(x))$. Here, $[x^n]$ is the operator that extracts the coefficient of $x^n$ \cite{MC}. We often denote this matrix by $[g(x), f(x)]$.
Associated with the pair $(g(x), f(x))$ of power series are two other power series,
$$A(x)=f'(\bar{f}(x))$$ and
$$Z(x)=\frac{g'(\bar{f}(x))}{g(\bar{f})(x)},$$ where
the power series $\bar{f}(x)$ is the compositional inverse or reversion of $f(x)$. Thus we have
$$f(\bar{f}(x))=x,\quad\quad \bar{f}(f(x))=x.$$
The matrix $P_M$ with bivariate generating function
$$e^{xy}(Z(x)+yA(x))$$ is called the production matrix of $M$. (Note that $A(x)$ and $Z(x)$ are also referred to as $r(x)$ and $c(x)$ in the literature).  It is equal to
$$P_M=M^{-1} \bar{M}$$ where $\bar{M}$ is the matrix $M$ with its top row removed.
The central fact that we use in this note is the following.
If $Z(x)$ and $A(x)$ are of the form
$$Z(x)=\alpha+\beta x, \quad\quad A(x)=1+\gamma x + \delta x^2,$$ then the production matrix will be tri-diagonal, corresponding to the family of orthogonal polynomials $P_n(x)$ that satisfy the three-term recurrence
$$P_n(x)=(x-(\alpha+(n-1)\gamma)P_{n-1}(x)-((n-1)\beta+(n-1)(n-2)\delta)P_{n-2}(x),$$
with $P_0(x)=1$ and $P_1(x)=x-\alpha$. The inverse matrix $M^{-1}$ is then the coefficient array of these polynomials. Thus if $m_{n,k}^*$ is the general $(n,k)$-th element of $M^{-1}$, we have
$$P_n(x)=\sum_{k=0}^n m_{n,k}^* x^k.$$

The first column elements of $M$ then represent the moments of the family of orthogonal polynomials $P_n(x)$.

\section{Dowling polynomials and Tanny-Dowling polynomials as moments}
We define the Whitney numbers $w_m(n,k)$ and $W_m(n,k)$ of the first and second kind, respectively, of Dowling lattices, by
$$\sum_{n=0}^{\infty} w_m(n,k)\frac{z^n}{n!}=\frac{(1+mz)^{-\frac{1}{m}}(\ln(1+mz))^k}{m^k k!},$$ and
$$\sum_{n=0}^{\infty} W_m(n,k)\frac{x^n}{n!}=\frac{e^z}{m^k k!}(e^{mz}-1)^k.$$
In terms of exponential Riordan arrays, this means that the Whitney numbers of the first kind $w_m(n,k)$ of the Dowling lattices are the elements of the exponential Riordan array
$$\left[\frac{1}{(1+mz)^\frac{1}{m}}, \frac{1}{m}\ln(1+mz)\right].$$
Similarly, the Whitney numbers of the second kind $W_m(n,k)$ are elements of the inverse exponential Riordan array
$$\left[e^z, \frac{1}{m} (e^{mz}-1)\right]=\left[\frac{1}{(1+mz)^\frac{1}{m}}, \frac{1}{m}\ln(1+mz)\right]^{-1}.$$
Explicitly, we have
$$w_m(n,k)=\sum_{i=0}^n (-1)^{i-k}\binom{i}{k}m^{n-i}s(n,i),$$ and
$$W_m(n,k)=\sum_{i=k}^n \binom{n}{i}m^{i-k}S(i,k)=\frac{1}{m^k k!}\sum_{i=0}^k \binom{k}{i}(-1)^{k-i}(mi+1)^n,$$ where $s(n,k)$ and $S(n,k)$ are the Stirling numbers of the first and second kind, respectively.

We have
$$S(n,k)=\frac{1}{k!} \sum_{j=0}^k (-1)^{k-j}\binom{k}{j}j^n.$$

Now the Stirling numbers $S(n,k)$ of the second kind are the elements of the exponential Riordan array
$$\left[1, e^x-1\right].$$ Hence we obtain that the Whitney numbers of the second kind are the elements of the
exponential Riordan array
$$\left[e^x, x\right]\cdot \left[1, e^{mx}-1\right]= \left[e^x,\frac{1}{m}( e^{mx}-1)\right].$$

Note that when $m=0$, we have $W_0(n,k)=\binom{n}{k}$ with corresponding exponential Riordan array the binomial matrix $[e^x,x]$.

\begin{definition} The Dowling polyomials $D_m(n,x)$ are defined by
$$D_m(n,x)=\sum_{k=0}^n W_m(n,k)x^k.$$
\end{definition}

\begin{lemma}
We have
$$D_m(n,x)=[x^n] e^t e^{\frac{(e^{mt}-1)x}{m}}.$$
\end{lemma}
\begin{proof}
Regarded as an infinite vector, the sequence $(D_m(n,x))_{n\ge 0}$ has generating function given by
$$  \left[e^t,\frac{1}{m}( e^{mt}-1)\right].e^{tx}=e^t e^{\frac{(e^{mt}-1)x}{m}}.$$
\end{proof}
Thus we have
$$\sum_{n=0}^{\infty} D_m(n,x)\frac{z^n}{n!} = e^z e^{\frac{(e^{mz}-1)x}{m}}.$$

\begin{proposition} The Dowling polynomials $D_m(n,x)$ form the moments for a family of orthogonal polynomials.
\end{proposition}
\begin{proof} The exponential Riordan array
$$\left[e^z e^{\frac{(e^{mz}-1)x}{m}}, \frac{1}{m}(e^{mz}-1)\right]$$ has a tri-diagonal production matrix with generating function
$$e^{zy}(1+x+xmz+y(1+mz)).$$
Thus the polynomial sequence $D_m(n,x)$ constitutes the moment sequence for the polynomials $P_n^{(m)}(z)$ that have coefficient array
$$\left[e^z e^{\frac{(e^{mz}-1)x}{m}}, \frac{1}{m}(e^{mz}-1)\right]^{-1}=\left[\frac{e^{-xz}}{(1+mz)^{\frac{1}{m}}}, \frac{1}{m}\ln(1+mz)\right].$$
The orthogonal polynomials $P_n^{(m)}(z)$ then satisfy the following three term recurrence.
$$P_n^{(m)}(z)=(z-(1+x+(n-1)m))P_{n-1}^{(m)}(z)-(n-1)mx P_{n-2}^{(m)}(z),$$ with
$P_0^{(m)}(z)=1$ and $P_1^{(m)}(z)=z-(1+x)$.
\end{proof}

\begin{definition} The Tanny-Dowling polynomials $\mathcal{F}_m(n,x)$ are defined by
$$\mathcal{F}_m(n,x)=\sum_{k=0}^n k! W_m(n,k)x^k.$$
\end{definition}

\begin{proposition} The Tanny-Dowling polynomials $\mathcal{F}_m(n,x)$ are the moments for a family of orthogonal polynomials.
\end{proposition}
\begin{proof}
The Tanny-Dowling polynomials $\mathcal{F}_m(n,x)$ have a generating function given by
$$\frac{m e^z}{m+x-xe^{mz}}.$$
The exponential Riordan array
$$M=\left[\frac{m e^z}{m+x-xe^{mz}}, \frac{e^{mz}-1}{m+x-xe^{mz}}\right]$$ has a tri-diagonal production matrix, with bivariate generating function
$$e^{yz}(1+x+x(x+m)z+y(1+(2x+m)z+x(x+m)z^2)).$$ This implies that the polynomials $\mathcal{F}_m(n,x)$ are the moment sequence for the family of orthogonal polynomials whose coefficient array is given by
$$\left[\frac{m e^z}{m+x-xe^{mz}}, \frac{e^{mz}-1}{m+x-xe^{mz}}\right]^{-1}=\left[\frac{(1+xz)^{\frac{1-m}{m}}}{(1+z(m+x))^{1/m}}, \frac{1}{m}\ln\left(\frac{1+z(m+x)}{1+zx}\right)\right].$$
The corresponding family of orthogonal polynomials $Q_n^{(m)}(z)$  satisfies the following three-term recurrence.
$$Q_n^{(m)}(z)=(z-(1+(2n-1)x+(n-1)m))Q_{n-1}^{(m)}(z)-(n-1)^2 x(x+m)Q_{n-2}^{(m)}(z),$$ with
$Q_0^{(m)}(z)=1$ and $Q_1^{(m)}(z)=z-(1+x)$.
\end{proof}
\section{The Eulerian-Dowling polynomials as moments}
 The Eulerian-Dowling polynomials are defined \cite{Rahmani} as follows.
\begin{definition} The Eulerian-Dowling polynomials $\mathcal{A}_m(n,x)$ are defined by
$$\mathcal{A}_m(n,x)=\sum_{k=0}^n k! W_m(n,k)(x-1)^{n-k}.$$
\end{definition}

\begin{proposition}
The Eulerian-Dowling polynomials $\mathcal{A}_m(n,x)$ are the moments for a family of orthogonal polynomials.
\end{proposition}
\begin{proof} It is known \cite{Rahmani} that the Eulerian-Dowling polynomials have generating function given by
$$\frac{m(1-x)e^{z(x-1)}}{e^{mz(x-1)}-(mx-m+1)}.$$
However, the exponential Riordan array
$$\left[\frac{m(1-x)e^{z(x-1)}}{e^{mz(x-1)}-(mx-m+1)}, \frac{1-e^{mz(x-1)}}{e^{mz(x-1)}-(mx-m+1)}\right]$$ has a production matrix with generating function
$$e^{yz}(x+(m(x-1)+2)z+y(1+(m(x-1)+2)z+(m(x-1)+1)z^2)),$$
that is tri-diagonal. Thus the Eulerian-Dowling polynomials are moments for the orthogonal polynomials that have
$$\left[\frac{m(1-x)e^{z(x-1)}}{e^{mz(x-1)}-(mx-m+1)}, \frac{1-e^{mz(x-1)}}{e^{mz(x-1)}-(mx-m+1)}\right]^{-1}$$ or
$$\left[\frac{(1+z)^{\frac{1-m}{m}}}{(1+z(1-m)+mxz)^{1/m}}, \frac{\ln(1+z(1-m)+mxz)-\ln(1+z)}{m(x-1)}\right]$$ as coefficient array. These polynomials therefore satisfy the following three-term recurrence.
$$P_n(z)=(z-(x+(n-1)(m(x-1)+2))P_{n-1}(z)-((n-1)^2 m(x-1)+m(n-1))P_{n-2}(z),$$ with
$P_0(z)=1$, $P_1(z)=z-x$.
\end{proof}

\section{Continued fractions and Hankel transforms}
The Eulerian-Dowling polynomials $\mathcal{A}_m(n,x)$, as moments, have the following continued fraction expression for their generating function $\sum_{n=0}^{\infty}\mathcal{A}_m(n,x)z^n$.
\begin{scriptsize}
$$ \cfrac{1}{1-xz-
\cfrac{(m(x-1)+1)z^2}{1-(x(m+1)-(m-2))z-
\cfrac{4(m(x-1)+1)z^2}{1-(x(2m+1)-2(m-2))z-
\cfrac{9(m(x-1)+1)z^2}{1-(x(3m+1)-3(m-2))z-\cdots}}}}.$$
\end{scriptsize}
In particular, we obtain that the Hankel transform of the sequence $\mathcal{A}_m(n,x)$ is given by
$$h_n = (m(x-1)+1)^{\binom{n+1}{2}} \prod_{k=0}^n k!^2.$$

In a similar fashion, we see that the Dowling polynomials $D_m(n,x)$ have a generating function given by
$$\cfrac{1}{1-(x+1)z-
\cfrac{mxz^2}{1-(x+m+1)z-
\cfrac{2mxz^2}{1-(x+2m+1)z-
\cfrac{3mxz^2}{1-(x+3m+1)z-\cdots}}}}.$$
This implies that the Hankel transform of the Dowling polynomials is given by
$$h_n = (mx)^{\binom{n+1}{2}} \prod_{k=0}^n k!.$$

The Tanny-Dowling polynomials $\mathcal{F}_m(n,x)$ have a generating function given by
$$\cfrac{1}{1-(x+1)z-
\cfrac{x(x+m)z^2}{1-(3x+m+1)z-
\cfrac{4x(x+m)z^2}{1-(5x+2m+1)z-
\cfrac{9x(x+m)z^2}{1-(7x+3m+1)z-\cdots}}}}.$$
Thus the Hankel transform of the sequence $\mathcal{F}_m(n,x)$ is given by
$$h_n =(x(x+m))^{\binom{n+2}{2}}\prod_{k=0}^n k!^2.$$

\section{Bivariate geometric polynomials and the Tanny-Dowling polynomials}
The geometric polynomials $\omega_n(x)$ are defined by
$$\omega_n(x)=\sum_{k=0}^n k! S(n,k)x^k.$$
We define the bivariate geometric polynomials $\omega_n(x,y)$ by
$$\omega_n(x,y)=\sum_{k=0}^n k! S(n,k)x^ky^{n-k}.$$
We then have the following result.
\begin{proposition}
The bivariate geometric polynomials $\omega_n(z,m)$ are moments for the family of orthogonal polynomials that have coefficient array given by the exponential Riordan array
$$\left[\frac{1}{1+zx}, \frac{1}{m}\ln\left(\frac{1+z(m+x)}{1+xz}\right)\right].$$
\end{proposition}
We have
$$\left[\frac{1}{1+zx}, \frac{1}{m}\ln\left(\frac{1+z(m+x)}{1+xz}\right)\right]^{-1}=\left[\frac{m}{m+x-xe^{mz}}, \frac{e^{mz}-1}{m+x-xe^{mz}}\right].$$
Thus the bivariate geometric polynomials $\omega_n(x,m)$ have exponential generating function
$$\frac{m}{m+x-xe^{mz}}.$$
Using the production matrix of the moment array $\left[\frac{m}{m+x-xe^{mz}}, \frac{e^{mz}-1}{m+x-xe^{mz}}\right]$ we also find that the bivariate geometric polynomials $\omega_n(x,m)$ have a generating function given by the following continued fraction.
$$\cfrac{1}{1-
\cfrac{xz}{1-
\cfrac{(x+m)z}{1-
\cfrac{2xz}{1-
\cfrac{2(x+m)z}{1-
\cfrac{3xz}{1-
\cfrac{3(x+m)z}{1-\cdots}}}}}}},$$ or
$$\cfrac{1}{1-xz-
\cfrac{x(x+m)z^2}{1-(3x+m)z-
\cfrac{4x(x+m)z^2}{1-(5x+2m)z-
\cfrac{9x(x+m)z^2}{1-(7x+3m)z-\cdots}}}}.$$

In particular, the Hankel transform of the bivariate geometric polynomials is given by 
$$h_n = (x(x+m))^{\binom{n+1}{2}} \prod_{k=0}^n k!^2.$$ 

Now since
$$[e^z, z] \cdot \left[\frac{m}{m+x-xe^{mz}}, \frac{e^{mz}-1}{m+x-xe^{mz}}\right]=\left[\frac{me^z}{m+x-xe^{mz}}, \frac{e^{mz}-1}{m+x-xe^{mz}}\right],$$ we obtain the following result.
\begin{corollary}
The Tanny-Dowling polynomials $\mathcal{F}_m(n,x)$ are given by the binomial transform of the bivariate geometric polynomials $\omega_n(x,m)$. That is,
$$\mathcal{F}_m(n,x)=\sum_{k=0}^n \binom{n}{k}\omega_k(x,m).$$
\end{corollary}

We can define a modified version of the bivariate geometric polynomials as follows.
$$\tilde{\omega}_n(x,m)=\sum_{k=0}^n (k+1)!S(n,k)x^k m^{n-k}.$$
We can then show the following.
\begin{proposition} The modified bivariate geometric polynomials $\tilde{\omega}_n(x,m)$ are the moments for the family of orthogonal polynomials which have their coefficient array given by the exponential Riordan array
$$\left[\frac{1}{(1+zx)^2}, \frac{1}{m}\ln\left(\frac{1+z(m+x)}{1+xz}\right)\right].$$
\end{proposition}
We have
$$\left[\frac{1}{1+zx}, \frac{1}{m}\ln\left(\frac{1+z(m+x)}{1+xz}\right)\right]^{-1}=\left[\left(\frac{m}{m+x-xe^{mz}}\right)^2, \frac{e^{mz}-1}{m+x-xe^{mz}}\right].$$
Thus the modified bivariate geometric polynomials $\tilde{\omega}_n(x,m)$ have exponential generating function
$$\frac{m^2}{(m+x-xe^{mz})^2}.$$

Using the production matrix of the moment array $\left[\frac{m^2}{(m+x-xe^{mz})^2}, \frac{e^{mz}-1}{m+x-xe^{mz}}\right]$ we also find that the modified bivariate geometric polynomials $\tilde{\omega}_n(x,m)$ have a generating function given by the following continued fraction.
$$\cfrac{1}{1-
\cfrac{2xz}{1-
\cfrac{(x+m)z}{1-
\cfrac{3xz}{1-
\cfrac{2(x+m)z}{1-
\cfrac{4xz}{1-
\cfrac{3(x+m)z}{1-\cdots}}}}}}},$$ or
$$\cfrac{1}{1-2xz-
\cfrac{2x(x+m)z^2}{1-(4x+m)z-
\cfrac{6x(x+m)z^2}{1-(6x+2m)z-
\cfrac{12x(x+m)z^2}{1-(8x+3m)z-\cdots}}}}.$$

In particular, these polynomials have a Hankel transform given by 
$$h_n = (x(x+m))^{\binom{n+1}{2}} \prod_{k=0}^n ((k+1)(k+2))^{n-k}=(x(x+m))^{\binom{n+1}{2}} (n+1)! \prod_{k=0}^n k!^2.$$ 

\section{Example sequences}
The polynomials that we have studied give rise to many known and interesting integer sequences by taking different values of the variable $x$ and the parameter $m$. The following is a small selection of these. 
We refer to these example sequences by their entry number in the On-Line Encyclopedia of Integer Sequences \cite{SL1, SL2}.

\begin{example} \textbf{The Dowling polynomials}.

The sequence $D_1(n,1)$ is the sequence \seqnum{A000110}$(n+1)$ beginning
$$1, 2, 5, 15, 52, 203, 877, 4140, 21147, 115975, 678570,\ldots.$$ These are the once shifted Bell numbers.

The sequence $D_2(n,1)$ is the sequence \seqnum{A007405} of the Dowling numbers
$$1, 2, 6, 24, 116, 648, 4088, 28640, 219920, 1832224, 16430176,\ldots.$$

The sequence $D_1(n,3)$ is the sequence \seqnum{A035009}$(n+1)$ which is the Stirling transform of the binomial transform of the natural numbers. This sequence begins
$$1, 3, 11, 47, 227, 1215, 7107, 44959, 305091, 2206399, 16913987,\ldots.$$

\end{example}
\begin{example} \textbf{The Tanny-Dowling polynomials}.

The sequence $\mathcal{F}_0(n,1)$ is the sequence \seqnum{A000522} that counts the total number of arrangements of a set of $n$ elements. This sequence begins
$$1, 2, 5, 16, 65, 326, 1957, 13700, 109601, 986410, 9864101,\ldots.$$

The sequence $\mathcal{F}_1(n,1)$ is the sequence \seqnum{A000629} that counts the number of necklaces of partitions of $n+1$ labeled beads. This sequence begins
$$1, 2, 6, 26, 150, 1082, 9366, 94586, 1091670, 14174522, 204495126,\ldots.$$

The sequence $\mathcal{F}_0(n,2)$ is \seqnum{A010844}, which counts the number of ways to sort a spreadsheet with $n$ columns. This sequence begins
$$1, 3, 13, 79, 633, 6331, 75973, 1063623, 17017969, 306323443, 6126468861, 134782314943,\ldots.$$

The sequence $\mathcal{F}_1(n,2)$ is \seqnum{A004123}, which counts the number of generalized weak orders on $n$ points. This sequence begins
$$1, 2, 10, 74, 730, 9002, 133210, 2299754, 45375130, 1007179562,\ldots.$$
\end{example}

\begin{example} \textbf{The Eulerian-Dowling polynomials}

The sequence $\mathcal{A}_0(n,2)$ is the sequence \seqnum{A000522} that counts the total number of arrangements of a set with $n$ elements.

The sequence $\mathcal{A}_{-2}(n,2)$ is the sequence \seqnum{A119880} with e.g.f. $e^{2x} \sech(x)$. This coincides with the values of the Swiss-Knife polynomials (\seqnum{A153641}) at $x=2$. This sequence begins
$$1, 2, 3, 2, -3, 2, 63, 2, -1383, 2, 50523, 2, -2702763, 2, 199360983,\ldots.$$

The sequence $\mathcal{A}_1(n,3)$ is \seqnum{A123227}, or $2^{n+1} \polylog(-n,1/3)$. This sequence begins
$$1, 3, 12, 66, 480, 4368, 47712, 608016, 8855040, 145083648,\ldots.$$
\end{example}

\bigskip
\hrule
\bigskip
\noindent 2010 {\it Mathematics Subject Classification}: Primary
15B36; Secondary 33C45, 11B83, 11C20, 05A15.
\noindent \emph{Keywords:} exponential Riordan array, Dowling polynomials, Tanny-Dowling polynomials, Eulerian-Dowling polynomials,  moment sequence, Hankel transform

\bigskip
\hrule
\bigskip
\noindent Concerns sequences
\seqnum{A000110},
\seqnum{A000522},
\seqnum{A000629},
\seqnum{A004123},
\seqnum{A010844},
\seqnum{A035009},
\seqnum{A119880},
\seqnum{A123227}.

\end{document}